\documentclass[11pt]{amsart}

\usepackage[T1]{fontenc}
\usepackage[utf8]{inputenc}

\usepackage{amssymb,mathtools}

\usepackage{microtype}

\usepackage{enumitem}
\usepackage{subcaption}
\usepackage[colorlinks=true,linkcolor=blue,citecolor=blue,urlcolor=blue]{hyperref}

\usepackage{dsfont}
\usepackage{mathrsfs}

\newtheorem{theorem}{Theorem}[section]

\newtheorem{lemma}[theorem]{Lemma}
\newtheorem{corollary}[theorem]{Corollary}
\theoremstyle{definition}
\newtheorem{definition}[theorem]{Definition}
\theoremstyle{remark}
\newtheorem{remark}[theorem]{Remark}

\theoremstyle{definition}
\newtheorem{example}{Example}[section]

\newcommand{\R}{\mathbb{R}}
\newcommand{\E}{\mathbb{E}}
\newcommand{\MoreauYosida}[2]{\operatorname{e}_{#2} #1}
\newcommand{\Prox}[2]{\operatorname{prox}_{#2 #1}}
\newcommand{\proj}{\operatorname{proj}}
\newcommand{\prox}{\operatorname{prox}}

\newcommand{\dom}{\operatorname{dom}}

\title[Rates for Stochastic Proximal and Projection Estimators]{Convergence Rates for Stochastic Proximal and Projection Estimators}

\author{Diego Morales}
\address{Departamento de Ingenier\'ia Matem\'atica, Universidad de Chile, Santiago, Chile}
\email{dhmorales@dim.uchile.cl}

\author{Pedro P\'erez-Aros}
\address{Departamento de Ingenier\'ia Matem\'atica and Center for Mathematical Modeling (CMM), Universidad de Chile, Santiago, Chile}
\email{pperez@dim.uchile.cl}

\author{Emilio Vilches}
\address{Instituto de Ciencias de la Ingenier\'ia, Universidad de O'Higgins, Rancagua, Chile; and Center for Mathematical Modeling (CMM), Universidad de Chile, Santiago, Chile}
\email{emilio.vilches@uoh.cl}


\date{\today}

\begin{document}

\begin{abstract}
In this paper, we establish explicit convergence rates for the stochastic smooth approximations of infimal convolutions introduced and developed in \cite{MR4581306,MR4923371}. In particular, we quantify the convergence of the associated barycentric estimators toward proximal mappings and metric projections. We prove a dimension-explicit $\sqrt{\delta}$ bound, with explicit constants for the proximal mapping, in the $\rho$-weakly convex (possibly nonsmooth) setting, and we also obtain a dimension-explicit $\sqrt{\delta}$ rate for the metric projection onto an arbitrary convex set with nonempty interior. Under additional regularity, namely $C^{2}$ smoothness with globally Lipschitz Hessian, we derive an improved linear $O(\delta)$ rate with explicit constants, and we obtain refined projection estimates for convex sets with local $C^{2,1}$ boundary. Examples demonstrate that these rates are optimal.
\end{abstract}

\maketitle

\section{Introduction}

Proximal mappings and metric projections are fundamental concepts in variational analysis and optimization. Introduced by J.-J.~Moreau in \cite{MR201952}, proximal mappings have since become central tools in modern optimization, notably through proximal algorithms and operator-splitting methods and their many variants (see, e.g., \cite{MR3616647}). Given a proper function \(f:\R^n\to\R\cup \{+\infty\}\) and \(\lambda>0\), the proximal point \(\prox_{\lambda f}(x)\) is the unique minimizer
of \(y\mapsto f(y)+\frac{1}{2\lambda}\|x-y\|^2\) whenever \(f\) is \(\rho\)-weakly convex and \(\lambda\in(0,1/\rho)\).
When \(f=\iota_C\) is the indicator of a closed convex set \(C\), the proximal mapping reduces to the metric projection
\(\proj_C(x)\). In many applications, however, evaluating \(\prox_{\lambda f}(x)\) or \(\proj_C(x)\) exactly is computationally
expensive, and in black-box settings only function values (possibly noisy) are available.

This paper studies a family of \emph{stochastic} (zeroth-order) estimators for proximal points and projections based on
Gaussian perturbations and exponential reweighting. For \(\delta>0\) and $\lambda>0$, we define the barycenter
\begin{equation}\label{eq:intro_mdelta}
m_\delta(x)
:=\frac{\E_{Y\sim\mathcal{N}(x,\delta\lambda I)}\!\big[Y\,\exp(-f(Y)/\delta)\big]}
{\E_{Y\sim\mathcal{N}(x,\delta\lambda I)}\!\big[\exp(-f(Y)/\delta)\big]}.
\end{equation}
When \(f=\iota_C\), the estimator becomes the conditional mean
\[
p_\delta(x)=\E\big[Y\mid Y\in C\big],\qquad Y\sim\mathcal{N}(x,\delta I),
\]
which can be interpreted as a ``smoothed projection'' onto \(C\).

The estimators \eqref{eq:intro_mdelta} arise naturally from Laplace-type asymptotics for integrals with kernel
\(\exp(-g/\delta)\), where \(g(y):=f(y)+\frac{1}{2\lambda}\|x-y\|^2\). Two recent strands of work developed this idea. First, motivated by Hamilton-Jacobi equations, \cite{MR4581306} proposed the HJ-Prox method and
proved that, for \(\rho\)-weakly convex objectives under mild regularity, the barycenter converges to the proximal point as \(\delta\to0^+\).
Second, \cite{MR4923371} places the relationship between infimal convolutions and self-normalized Laplace approximation
front and center, extending the construction beyond the quadratic kernel and establishing asymptotic validity under
weak local assumptions; in particular, their framework includes smoothed set projections
\(p_{\delta}(x)=\E[Y\mid Y\in C]\) and proves \(p_{\delta(x)}\to \operatorname{proj}_C(x)\) as \(\delta\to0^+\) under mild local geometric
conditions near the projection point.

While these references provide a compelling unifying viewpoint and broad asymptotic consistency results, they leave open
a basic quantitative question that is decisive for algorithmic use: \emph{how fast} do these stochastic estimators converge,
and how do the constants depend on the ambient dimension and on the geometry/smoothness of the problem?

\subsection*{Our contributions: explicit and sharp convergence rates}
The main goal of this work is to provide \emph{non-asymptotic}, \emph{dimension-explicit} convergence rates for
\eqref{eq:intro_mdelta} and its projection specialization, and to identify regimes where improved rates hold.

\noindent\textbf{(i) Nonsmooth weakly convex case: a sharp \(\sqrt{\delta}\) rate.}
Assume \(f\) is proper, lower semicontinuous, and \(\rho\)-weakly convex, and fix \(\lambda\in(0,1/\rho)\) with
\(\mu:=\frac{1}{\lambda}-\rho>0\). Under the mild geometric condition that \(\dom f\) has nonempty interior,
we prove the global bound (valid for every \(\delta>0\))
\[
\|m_\delta(x)-\prox_{\lambda f}(x)\|\le \sqrt{\frac{n\delta}{\mu}}.
\]
This estimate is complemented by explicit examples showing that the \(\sqrt{\delta}\) order cannot be improved in general.

\noindent\textbf{(ii) Smooth case: an \(O(\delta)\) rate with explicit constants.}
If \(f\in C^2(\R^n)\) is \(\rho\)-weakly convex and \(\nabla^2 f\) is globally Lipschitz with constant \(L\), we establish the
linear-rate refinement
\[
\|m_\delta(x)-\prox_{\lambda f}(x)\|\le \frac{nL}{\mu^2}\,\delta,
\qquad \mu=\frac{1}{\lambda}-\rho,
\]
again for every \(\delta>0\). This result quantifies a phenomenon suggested by the smoothing interpretations in
\cite{MR4581306,MR4923371}: additional curvature regularity improves the bias of the barycentric estimator.

\noindent\textbf{(iii) Projection corollaries and boundary-sensitive refinements.}
Taking \(f=\iota_C\) and $\lambda=1$ yields a quantitative convergence estimate for convex projections:
\[
\|p_\delta(x)-\proj_C(x)\|\le \sqrt{n\,\delta}.
\]
Beyond the convex Lipschitz regime, we analyze the case of convex sets with \(C^{2,1}\) boundary charts at the projection point and obtain a refined \(O(\delta)\) expansion.
This provides a bridge between the purely asymptotic projection consistency statement in \cite{MR4923371} and the boundary-geometry-dependent expansions needed in fine asymptotic analysis.

The remainder of the paper is organized as follows.
Section~\ref{preliminaries} collects the mathematical preliminaries used throughout.
Section~\ref{Section3} establishes the \(\sqrt{\delta}\) bound for nonsmooth weakly convex functions,
provides sharpness examples, and derives the corollary for convex projections.
Section~\ref{Section4} proves the \(O(\delta)\) rate under \(C^2\) regularity with Lipschitz continuous Hessian
and develops a refined analysis for convex sets with \(C^{2,1}\) boundaries. The paper ends with some concluding remarks and possible lines of future research.

\section{Mathematical Preliminaries}\label{preliminaries}
Throughout this paper, we work in the finite-dimensional Hilbert space $\mathbb{R}^n$ endowed with the inner product $\langle \cdot, \cdot\rangle$ and the associated norm $\Vert \cdot\Vert$. The closed unit ball in $\mathbb{R}^n$ is denoted by $\mathbb{B}$. The open ball of radius $R$ centered at $x$ is denoted by $B_R(x)$. For simplicity, the open ball of radius $R$ centered at $0$ is denoted by $B_R$. Moreover, $\mathbb{S}^{n-1}$ denotes the unit sphere in $\mathbb{R}^n$.

We use the convention $\exp(+\infty) =+\infty $ and $\exp( -\infty ) = 0$.  

\noindent For a set $A\subset \mathbb{R}^n$, the \emph{indicator function} of $A$ is defined by  
\begin{equation*}
\iota_A(x):=\left\{ \begin{array}{cc}
0 &\text{ if } x\in A,\\
+\infty &\text{ if } x\notin A.
\end{array}\right.
\end{equation*}
The \emph{distance function} from $x\in \mathbb{R}^n$  to $A$ is given by $$d_{A}(x)=\inf\left\{  \|  x - y\| : y\in A\right\}.$$
 Moreover, if $A$ is nonempty, closed and convex, the \emph{projection} of $x\in \mathbb{R}^n$ onto $A$ is denoted by $\operatorname{proj}_{A} (x)$. Given a function $g\colon \mathbb{R}^n \rightarrow \mathbb{R}\cup \{+\infty \}$, the \emph{domain} of $g$ is
\begin{equation*}
\operatorname{dom}(g):=\{x\in \mathbb{R}^n: g(x)<+\infty \}.
\end{equation*}
We say that $g$ is \emph{proper} if $\operatorname{dom}g\neq
\emptyset$. Given $\rho\geq 0$, a function $g\colon \R^n\to \mathbb{\R}\cup \{+\infty \}$ is called $\rho$-weakly convex if $g+\frac{\rho}{2}\|\cdot\|^2$ is convex. For $\mu>0$, we say that $g$ is $\mu$-strongly convex if $g-\frac{\mu}{2}\Vert \cdot\Vert^2$ is convex.

The Moreau envelope of index $\lambda>0$ of a function $g\colon \R^n\to \mathbb{R}\cup \{+\infty\}$ is denoted by $\MoreauYosida{g}{\lambda}\colon \mathbb{R}^n \to \mathbb{R}$ and is defined  by 
\begin{equation*}
    \MoreauYosida{g}{\lambda}(x):=\inf_{y\in \mathbb{R}^n} \left(	g(y) + \frac{1}{2\lambda } \|  x - y\|^2		\right) \textrm{ for all } x\in \mathbb{R}^n.
\end{equation*}
 If $g$ is proper, lower semicontinuous, and $\rho$-weakly convex, and if $0<\lambda<\frac{1}{\rho}$ (with the convention $1/0=+\infty$), then the above infimum is attained at a unique point $\Prox{g}{\lambda}(x)\in \mathbb{R}^n$, and
\begin{equation*}
\begin{aligned}
\MoreauYosida{g}{\lambda}(x)&=g(\Prox{g}{\lambda}(x))+\frac{1}{2\lambda}\Vert x-\Prox{g}{\lambda}(x)\Vert^2.
\end{aligned}
\end{equation*}
In this case, the operator $x\mapsto \Prox{g}{\lambda}(x)$ is everywhere defined and is called the \emph{proximal operator} of $g$ of index $\lambda$. Moreover, $\MoreauYosida{g}{\lambda}$ is continuously differentiable and, for each  $x\in \mathbb{R}^n$, 
\begin{equation*}
\nabla \MoreauYosida{g}{\lambda}(x)=\frac{1}{\lambda}(x-\Prox{g}{\lambda}(x)).
\end{equation*}

We refer to \cite{MR3288271,MR3616647} for more details.

Let $\delta>0$, and let $g\colon \mathbb{R}^n \to \mathbb{R}\cup \{+\infty\}$ be a proper function such that $e^{-g/\delta}$ is integrable and
$$
\int_{\mathbb{R}^n}e^{-g(y)/\delta}dy>0.
$$
Under these assumptions, $e^{-g/\delta}$ induces a probability measure. We denote by
$$
\mathbb{E}_{\sigma_{\delta}}[w(y)]=\frac{\int_{\mathbb{R}^n} w(y)e^{-g(y)/\delta}dy}{\int_{\mathbb{R}^n}e^{-g(y)/\delta}dy}
$$
the expectation of a given function $w$ with respect to the probability measure $\sigma_\delta$, which is formally defined for every Lebesgue-measurable set $A$ as
$$ \sigma_\delta (A) := \frac{\int_{A} e^{-g(y)/\delta}dy}{\int_{\mathbb{R}^n}e^{-g(y)/\delta}dy}.  $$

\section{Rates of Convergence for Nonsmooth Functions}\label{Section3}
Let \(f:\mathbb{R}^n\to \mathbb{R}\cup\{+\infty\}\) be a \(\rho\)-weakly convex function with \(\rho\ge 0\). 
The goal of this section is to establish an explicit convergence rate for the map
\[
m_{\delta}(x)
:=\frac{\mathbb{E}_{y\sim \mathcal{N}(x,\delta \lambda I)}\!\bigl[y\, \exp(-f(y)/\delta)\bigr]}
{\mathbb{E}_{y\sim \mathcal{N}(x,\delta \lambda I)}\!\bigl[\exp(-f(y)/\delta)\bigr]},
\]
and to quantify its approximation of the proximal point:
\[
m_{\delta}(x)\;\longrightarrow\;\operatorname{prox}_{\lambda f}(x)\qquad \text{as }\delta\downarrow 0.
\]
It is straightforward to verify that if $f$ is $\rho$-weakly convex for some $\rho\ge 0$ and $\operatorname{dom} f$ has nonempty interior, then $m_\delta$ is well-defined for every $\delta>0$ and $0<\lambda <1/\rho$.

The following result provides a convergence rate for this barycenter toward the proximal point.

\begin{theorem}\label{convergence}
   Let $f\colon \mathbb{R}^n \to \mathbb{R}\cup \{+\infty\}$ be proper, lower semicontinuous,  and $\rho$-weakly convex for some $\rho\geq 0$.  Assume that $\operatorname{dom}f$ has nonempty interior. Fix $x\in \mathbb{R}^n$ and $0<\lambda <1/\rho$, and set $\mu:=\frac{1}{\lambda}-\rho>0$. 
Then, for every $\delta>0$,
$$
\Vert m_{\delta}(x)-\operatorname{prox}_{\lambda f}(x)\Vert \leq \sqrt{\frac{n\delta}{\mu}}.
$$
\end{theorem}
\begin{proof} 
Fix $x\in \mathbb{R}^n$ and assume that $0<\lambda <1/\rho$. Consider 
\begin{equation*}
g(y):=f(y)+\frac{1}{2\lambda}\Vert x-y\Vert^2.
\end{equation*}
Since $f$ is $\rho$-weakly convex, it follows that $g$ is $\mu$-strongly convex and
\begin{equation}\label{growth}
    g(y)\geq g(\operatorname{prox}_{\lambda f}(x))+\frac{\mu}{2}\Vert y-\operatorname{prox}_{\lambda f}(x)\Vert^2 \textrm{ for all } y\in \mathbb{R}^n.
\end{equation}
Consider the vector field
$$
F(y):=(y-\operatorname{prox}_{\lambda f}(x))e^{-g(y)/\delta}.
$$
Now, let us show that  $\mathbb{E}_{\sigma_{\delta}}\left[\langle y-\operatorname{prox}_{\lambda f}(x),\nabla g(y)\rangle \right]\leq n\delta$.
 
Indeed, since $y\mapsto g(y)$ is proper, lsc, and convex, its domain is convex; in particular, $D:=\operatorname{dom}f=\operatorname{dom}g$ is convex with nonempty interior (by assumption). On $\operatorname{int}(D)$, every proper lsc convex function is locally Lipschitz, hence differentiable a.e. and belongs to $W^{1,1}_{\operatorname{loc}}(\operatorname{int}(D))$. Consequently, $\nabla g\in L^1_{\operatorname{loc}}(\operatorname{int}(D);\mathbb{R}^n)$, and for a.e. $y \in \operatorname{int}(D)$,
$$
\operatorname{div}F(y)=ne^{-g(y)/\delta}-\frac{1}{\delta}\langle y-\operatorname{prox}_{\lambda f}(x),\nabla g(y)\rangle e^{-g(y)/\delta}.
$$
Fix $R>0$ and set $\Omega_R:=\operatorname{int}(D) \cap B_R $. Since 
$$\Vert y-\operatorname{prox}_{\lambda f}(x)\Vert \leq R+\Vert \operatorname{prox}_{\lambda f}(x)\Vert =:C_R \textrm{ on } B_R
$$
and $e^{-g(y)/\delta}\leq e^{-\inf_{B_R}g/\delta}=:M_{R,\delta}<\infty$, it follows that for a.e. $y\in \Omega_{R}$,
$$
\vert \operatorname{div}F(y)\vert \leq nM_{R,\delta}+\frac{M_{R,\delta}C_{R}}{\delta}\Vert \nabla g(y)\Vert,
$$
Since $\nabla g\in L^1_{\operatorname{loc}}(\operatorname{int}(D);\mathbb{R}^n)$, we have $\nabla g\in L^1(\Omega_R;\mathbb{R}^n)$, and then $\operatorname{div}F\in L^1(\Omega_R)$. For $R>0$, apply the divergence theorem to $\Omega_R$:
\begin{equation*}
    \begin{aligned}
\int_{\Omega_R}\operatorname{div}F(y)dy&=\int_{\partial \Omega_R} F\cdot \nu\, d\mathcal{H}^{n-1}.
    \end{aligned}
\end{equation*}
Moreover, by Lemma \ref{co-area}, we can select $R_k\to \infty$ such that
\[
\int_{\partial \Omega_{R_k}} F \cdot \nu \, d\mathcal{H}^{n-1}
=
\underbrace{\int_{\partial D\cap {B}_{R_k}}
 F \cdot \nu_{D} \, d\mathcal{H}^{n-1}}_{=:I_1}+\underbrace{\int_{\partial {B}_{R_k}\cap D}
 F\cdot \nu_{{B}_{R_k}} \, d\mathcal{H}^{n-1}}_{=:I_2}.
\]
Now, on the one hand, since $D$ is convex and $\operatorname{prox}_{\lambda f}(x)\in D$, for $\mathcal{H}^{n-1}$-a.e. $y\in \partial D$ one has
$$
\langle y-\operatorname{prox}_{\lambda f}(x),\nu_D(y)\rangle \geq 0 \quad \Rightarrow \quad   F(y)\cdot \nu_D(y) \geq 0.
$$
Hence the $\partial D$-term  $I_1$ is nonnegative.  On the other hand, using \eqref{growth}, the term $I_2$  is bounded in absolute value by 
$$
CR_k^n \exp\left(-\frac{\mu}{2\delta}(R_k-\Vert \operatorname{prox}_{\lambda f}(x) \Vert)^2\right)\to 0 \textrm{ as } k\to \infty,
$$
hence the $\partial {B}_{R_k}$-term $I_2$ vanishes as $k\to \infty$. Therefore,
\begin{equation*}
\begin{aligned}
    0&\leq \limsup_{k\to \infty} \int_{ B_{R_k}\cap \partial D }  F\cdot \nu_D\, d\mathcal{H}^{n-1}\\
    &=\limsup_{k\to \infty}\Big(\int_{\partial D\cap  {B}_{R_k}  }  F\cdot \nu_D\, d\mathcal{H}^{n-1}+\int_{\partial {B}_{R_k}\cap D}  F\cdot \nu_{{B}_{R_k}}\, d\mathcal{H}^{n-1}\Big)\\
    &=\limsup_{k\to \infty}\int_{\partial \Omega_{R_k}} F\cdot \nu\, d\mathcal{H}^{n-1}\\
    &=\limsup_{k\to \infty}\int_{\Omega_{R_k}} \operatorname{div}F(y)dy\\
    &=\int_D \operatorname{div}F(y)dy\\
    &=n\int_D e^{-g(y)/\delta}dy-\frac{1}{\delta}\int_D \langle y-\operatorname{prox}_{\lambda f}(x),\nabla g(y)\rangle e^{-g(y)/\delta}dy.
\end{aligned}
\end{equation*}
Finally, dividing by $\int_{D}e^{-g(y)/\delta}dy$ and rearranging, we obtain the estimate. 
Next, by strong convexity and almost-everywhere differentiability, we obtain
$$
\langle \nabla g(y),y-\operatorname{prox}_{\lambda f}(x)\rangle \geq \mu \Vert y-\operatorname{prox}_{\lambda f}(x)\Vert^2.
$$
Taking expectation yields
$$
\mu \mathbb{E}_{\sigma_{\delta}}\left[\Vert y-\operatorname{prox}_{\lambda f}(x)\Vert^2\right]\leq\mathbb{E}_{\sigma_{\delta}}\left[\langle y-\operatorname{prox}_{\lambda f}(x),\nabla g(y)\rangle \right]=n\delta,
$$
hence $\mathbb{E}_{\sigma_{\delta}}\left[\Vert y-\operatorname{prox}_{\lambda f}(x)\Vert^2\right]\leq n\delta /\mu$. Finally, by Jensen's inequality, 
\begin{equation*}
    \begin{aligned}
  \Vert m_{\delta}(x)-\operatorname{prox}_{\lambda f}(x)\Vert^2&\leq \mathbb{E}_{\sigma_{\delta}}\left[\Vert y-\operatorname{prox}_{\lambda f}(x)\Vert\right]^2\\
  &\leq  \mathbb{E}_{\sigma_{\delta}}\left[\Vert y-\operatorname{prox}_{\lambda f}(x)\Vert^2\right]\\
  &\leq \frac{n\delta}{\mu}.    
    \end{aligned}
\end{equation*}
\end{proof}
\begin{remark}
Let \(Y\sim\sigma_\delta\), where \(\sigma_\delta\) has density proportional to \(\exp(-g/\delta)\).
The proof of Theorem~\ref{convergence} yields the \emph{mean-square localization} estimate
\begin{equation*}
\E\big[\|Y-\operatorname{prox}_{\lambda f}(x)\|^2\big]\le \frac{n\delta}{\mu}.
\end{equation*}
This bound also yields a simple concentration inequality: for every \(r>0\),
\begin{equation*}
\sigma_\delta\!\big(\|Y-\operatorname{prox}_{\lambda f}(x)\|\ge r\big)\le \frac{n\delta}{\mu\,r^2},
\end{equation*}
and therefore, for any \(\eta\in(0,1)\),
\begin{equation*}
\sigma_\delta\!\left(B_{\sqrt{n\delta/(\mu\eta)}}(\operatorname{prox}_{\lambda f}(x))\right)\ge 1-\eta.
\end{equation*}
\end{remark}

The following two examples show that our convergence rates are sharp.
\begin{example}
    Fix $n\geq 1$ and consider
    $$
    f(y)=\sum_{i=1}^n \max\{y_i,0\}.
    $$
    Then $f$ is proper, lower semicontinuous, convex and $\operatorname{dom}f=\mathbb{R}^n$ has nonempty interior. Fix $x=0$ and any $\lambda>0$, and set $\mu=1/\lambda$. Then, $\operatorname{prox}_{\lambda f}(0)=0$ and, by symmetry of $f$, for any $\delta>0$
    $$
    m_{\delta}(0)=(a_{\delta},\ldots,a_{\delta}),
    $$
    where 
    $$
    a_{\delta}:=\frac{\int_{\mathbb{R}}t\,e^{-\phi(t)/\delta}\,dt}{\int_{\mathbb{R}}e^{-\phi(t)/\delta}\,dt} \qquad \phi(t):=\max\{t,0\}+\frac{1}{2\lambda}t^2.
    $$
    Moreover, one can show that 
    $$
    a_{\delta}=-\sqrt{\frac{2}{\pi}\lambda \delta}+O(\delta) \textrm{ as } \delta \to 0^+.
    $$
    Consequently,
    $$
    \Vert m_{\delta}(0)-\operatorname{prox}_{\lambda f}(0)\Vert =\sqrt{n}\vert a_{\delta}\vert=\sqrt{\frac{2n}{\pi \mu}\delta}+O(\delta).
    $$
    This shows that the order of convergence in Theorem~\ref{convergence} is sharp.
\end{example}
The following example shows that the bound obtained in Theorem~\ref{convergence} is optimal.
\begin{example} Fix $n\geq 2$, take $\lambda>0$, set $\mu:=1/\lambda$, and let $x=0$. For $\alpha\in (0,\pi/2)$, define the circular cone
$$
K_{\alpha}:=\{y\in \mathbb{R}^n \colon \langle y,e_1\rangle \geq \Vert y\Vert\, \cos \alpha\},
$$
where $e_1$ denotes the first canonical basis vector. 
Let $f=\iota_{K_{\alpha}}$ be the indicator function of $K_{\alpha}$. Then $\operatorname{prox}_{\lambda f}(0)=0$. Moreover,
$$
\Vert m_{\delta}(0)-\operatorname{prox}_{\lambda f}(0)\Vert =\Vert m_{\delta}(0)\Vert=\sqrt{\frac{\delta}{\mu}} \mathbb{E}[\chi_n] \mathfrak{m}_{\alpha,n},
$$
where $\chi_n$ denotes the $\chi$-distribution with $n$ degrees of freedom and 
$$
\mathfrak{m}_{\alpha,n}:=\mathbb{E}\left[U_1 \mid U_1\geq \cos \alpha\right],
$$
with $U$ uniformly distributed on $\mathbb{S}^{n-1}$ and $U_1=\langle U,e_1\rangle$. In particular, since $U_1\in [\cos \alpha,1]$ on the conditioning event, we have $\mathfrak{m}_{\alpha,n}\in [\cos \alpha,1]$ and hence   $\mathfrak{m}_{\alpha,n}\to 1$ as $\alpha \searrow 0$. Furthermore, 
$$
\frac{\mathbb{E}[\chi_n]}{\sqrt{n}}\to 1 \textrm{ as  } n\to \infty.
$$
Consequently, for every $\varepsilon>0$ there exist $n$ and $\alpha$ such that, for all $\delta>0$, 
$$
\Vert m_{\delta}(0)-\operatorname{prox}_{\lambda f}(0)\Vert \geq (1-\varepsilon)\sqrt{\frac{n\delta}{\mu}}.
$$
This shows that the rate of convergence in Theorem~\ref{convergence} is sharp.
\end{example}
We end this section by showing that $m_{\delta}(x)$ coincides with $\operatorname{prox}_{\lambda f}(x)$ whenever $f$ is quadratic.
\begin{example}
Let $A$ be a symmetric and satisfy $A\succeq -\rho I$ for some $\rho\geq 0$, and define $$f(y)=\frac{1}{2}\langle Ay,y\rangle +\langle b,y\rangle +c.$$
Then for every $x\in \mathbb{R}^n$, every $\delta>0$, and every $0<\lambda <1/\rho$ (with the convention $1/0=+\infty$), one has
$$m_{\delta}(x)=\operatorname{prox}_{\lambda f}(x)=\left(A+\frac{1}{\lambda}I\right)^{-1}\left(\frac{1}{\lambda}x-b\right).
$$
\end{example}

\subsection{Stochastic Approximation of Projections onto Convex Sets}
Let $C \subset \mathbb{R}^n$ be a nonempty closed convex set with nonempty interior. We define
$$
 p_{\delta}(x):=\frac{\mathbb{E}_{y\sim \mathcal{N}(x,\delta  I)}[y\cdot  \mathbf{1}_C(y)]}{\mathbb{E}_{y\sim \mathcal{N}(x,\delta  I)}[\mathbf{1}_C(y)]}=\mathbb{E}_{y\sim \mathcal{N}(x,\delta I)}[y\mid y\in C]. 
$$
As a direct consequence of Theorem~\ref{convergence}, we obtain a quantified convergence result: 
$p_{\delta}(x)\to \operatorname{proj}_{C}(x)$ as $\delta\to 0^{+}$. 
\begin{corollary}\label{cor:convex-set}
Let $C\subset \mathbb{R}^n$ be a nonempty closed convex set with nonempty interior. 
Fix $x\in \mathbb{R}^n$. Then, for every $\delta>0$,
\[
\|p_{\delta}(x)-\operatorname{proj}_C(x)\|\le \sqrt{n\,\delta}.
\]
\end{corollary}

\section{Smooth results}\label{Section4}
In this section, we show that additional smoothness of the problem data yields improved convergence rates compared with those in Section \ref{Section3}. 

\begin{theorem}
Let $f\colon \mathbb{R}^n \to \mathbb{R}$ be $C^2$ and $\rho$-weakly convex with $\rho\geq 0$.  Assume that $\nabla^2 f$ is globally Lipschitz: there exists $L\geq 0$ such that
$$
\Vert \nabla^2 f(u)-\nabla^2 f(v)\Vert \leq L\Vert u-v\Vert \quad \textrm{ for all } u,v\in \mathbb{R}^n.
$$
Fix $x\in \mathbb{R}^n$ and $0<\lambda <1/\rho$, and set $\mu:=\frac{1}{\lambda}-\rho>0$. 
Then, for every $\delta>0$,
$$
\Vert m_{\delta}(x)-\operatorname{prox}_{\lambda f}(x)\Vert \leq \frac{nL}{\mu^2}\,\delta.
$$
\end{theorem}
\begin{proof} Fix $x\in \mathbb{R}^n$ and assume that $0<\lambda <1/\rho$. Let us consider 
\begin{equation*}
g(y):=f(y)+\frac{1}{2\lambda}\Vert x-y\Vert^2.
\end{equation*}
Since $f$ is $\rho$-weakly convex and $0<\lambda <1/\rho$, it follows that $g$ is $\mu$-strongly convex and that \eqref{growth} holds. In particular, \eqref{growth} implies that $e^{-g/\delta}$ is integrable for every $\delta>0$, so $m_{\delta}(x)$ is well-defined. \\
Given a function $w(\cdot)$, recall that
$$
\mathbb{E}_{\sigma_{\delta}}[w(y)]=\frac{\int_{\mathbb{R}^n} w(y)e^{-g(y)/\delta}dy}{\int_{\mathbb{R}^n}e^{-g(y)/\delta}dy}.
$$
\textbf{Claim 1:} $\mathbb{E}_{\sigma_{\delta}}[ \nabla g(y)]=0$.\\
\emph{Proof of Claim 1:} First, note that the Lipschitz continuity of the Hessian of $f$ implies that $\nabla g$ grows at most quadratically. More precisely, there exist constants $a,b>0$ such that
\[
\|\nabla g(y)\|\le a\|y\|^{2}+b \qquad \text{for all } y\in\mathbb{R}^{n}.
\]
In particular, $\nabla g$ is integrable with respect to the probability measure $\sigma_\delta$.

Second, we observe that
$$
\frac{\partial }{\partial y_i}\left(e^{-g(y)/\delta}\right)=-\frac{1}{\delta}\frac{\partial g}{\partial y_i} (y)e^{-g(y)/\delta}.
$$
Moreover, from \eqref{growth}, it follows that  
$$
0\leq e^{-g(y)/\delta}\leq e^{-g(\operatorname{prox}_{\lambda f}(x))/\delta}\cdot e^{-\frac{\mu}{2\delta}\Vert y-\operatorname{prox}_{\lambda f}(x)\Vert^2} \textrm{ for all } y\in \mathbb{R}^n,
$$
which implies that
$$
0=\int_{\mathbb{R}^n}  \frac{\partial }{\partial y_i}\left(e^{-g(y)/\delta}\right)dy=-\frac{1}{\delta} \int_{\mathbb{R}^n} \frac{\partial g}{\partial y_i} (y)e^{-g(y)/\delta} dy,
$$
which implies the claim.\\
\textbf{Claim 2:} Let $H=\nabla^2 g(\operatorname{prox}_{\lambda f}(x))$. Then, for all $y\in \mathbb{R}^n$
$$
\nabla g(y)=H(y-\operatorname{prox}_{\lambda f}(x))+r(y), 
$$
where $\Vert r(y)\Vert \leq \frac{L}{2}\Vert y-\operatorname{prox}_{\lambda f}(x)\Vert^2$.\\
\emph{Proof of Claim 2:} Since $\nabla g(\operatorname{prox}_{\lambda f}(x))=0$, the formula follows from Taylor's theorem with
$$
r(y)=\int_0^1 (\nabla^2 g(\operatorname{prox}_{\lambda f}(x)+t(y-\operatorname{prox}_{\lambda f}(x)))-H)dt.
$$
Because $\nabla^2 g=\nabla^2 f+\frac{1}{\lambda}I$ and $\nabla^2 f$ is $L$-Lipschitz, we have
$$
\Vert \nabla^2 g(u)-\nabla^2 g(v)\Vert \leq L\Vert u-v\Vert.
$$
Hence
$$
\Vert r(y)\Vert \leq \int_0^1 Lt\Vert y-\operatorname{prox}_{\lambda f}(x)\Vert^2=\frac{L}{2}\Vert y-\operatorname{prox}_{\lambda f}(x)\Vert^2,
$$
which proves the claim. \\
\textbf{Claim 3:} $m_{\delta}(x)-\operatorname{prox}_{\lambda f}(x)=-H^{-1}\mathbb{E}_{\sigma_{\delta}}[r(y)]$.\\
\emph{Proof of Claim 3:} From Claim 1 and 2 we have that
\begin{equation*}
\begin{aligned}
0&=\mathbb{E}_{\sigma_{\delta}}[ \nabla g(y)]=H(m_{\delta}(x)-\operatorname{prox}_{\lambda f}(x))+\mathbb{E}_{\sigma_{\delta}}[r(y)],
\end{aligned}
\end{equation*}
which implies the result.\\
\textbf{Claim 4:} For every $\delta>0$, we have 
$$\Vert m_{\delta}(x)-\operatorname{prox}_{\lambda f}(x)\Vert \leq \frac{nL}{\mu^2}\delta.$$
\emph{Proof of Claim 4:} On the one hand, from Claims 2 and 3,  we have
\begin{equation}\label{cota-inicial}
    \begin{aligned}
\Vert m_{\delta}(x)-\operatorname{prox}_{\lambda f}(x)\Vert &=\Vert H^{-1} \mathbb{E}_{\sigma_{\delta}}[r(y)]\Vert\\
&\leq \frac{1}{\mu}\mathbb{E}_{\sigma_{\delta}}\left[ \Vert r(y)\Vert \right] \\
&\leq \frac{L}{2\mu}\mathbb{E}_{\sigma_{\delta}}[\Vert y-\operatorname{prox}_{\lambda f}(x)\Vert^2],
    \end{aligned}
\end{equation}
where we used that $H\succcurlyeq \mu I$, hence $\Vert H^{-1}\Vert_2 \leq 1/\mu$, where $\Vert \cdot \Vert_2$ is the  induced $\ell_2$-operator norm.  Moreover, using that $\mathbb{E}_{\sigma_{\delta}}[y]=m_{\delta}(x)$, we get that
\begin{equation}\label{paralelogramo}
   \mathbb{E}_{\sigma_{\delta}}[\Vert y-\operatorname{prox}_{\lambda f}(x)\Vert^2]=\mathbb{E}_{\sigma_{\delta}}[\Vert y-m_{\delta}(x)\Vert^2]+\mathbb{E}_{\sigma_{\delta}}[\Vert m_{\delta}(x)-\operatorname{prox}_{\lambda f}(x)\Vert^2].
\end{equation}
On the other hand, we observe that $g/\delta$ is $\mu/\delta$ strongly convex. Hence, $h_{\delta}=e^{-g/\delta}$ is strongly log-concave. From Brascamp–Lieb inequality (see \cite[Theorem~4.1]{MR450480}), we get that {\small 
\begin{equation}\label{BL}
\begin{aligned}
\operatorname{Cov}_{\sigma_{\delta}}(y)\preceq \frac{\delta}{\mu}I \quad \Rightarrow  \quad  \mathbb{E}_{\sigma_{\delta}}[\Vert y-m_{\delta}(x)\Vert^2]=\operatorname{tr}\left(\operatorname{Cov}_{\sigma_{\delta}}(y)\right)\leq \operatorname{tr}\left(\frac{\delta}{\mu}I\right)=\frac{n\delta}{\mu},
\end{aligned}
\end{equation}}
where $\operatorname{Cov}_{\sigma_\delta}$ denotes the covariance matrix of a random vector under the probability measure $\sigma_\delta$. Moreover, Theorem~\ref{convergence} gives us the estimation
\begin{equation*}
\Vert m_{\delta}(x)-\operatorname{prox}_{\lambda f}(x)\Vert^2 \leq \frac{n\delta}{\mu}.
\end{equation*}
Plugging \eqref{BL} and the above inequality  into \eqref{paralelogramo} yields
\begin{equation*}
    \mathbb{E}_{\sigma_{\delta}}[\Vert y-\operatorname{prox}_{\lambda f}(x)\Vert^2]\leq \frac{n\delta}{\mu}+\frac{n\delta}{\mu}=\frac{2n\delta}{\mu}.
\end{equation*}
Finally, from the above inequality and \eqref{cota-inicial}, we get
$$
\Vert m_{\delta}(x)-\operatorname{prox}_{\lambda f}(x)\Vert \leq \frac{L}{2\mu}\cdot \frac{2n\delta}{\mu}=\frac{nL}{\mu^2}\delta,
$$
which proves Claim  4 and the theorem.
\end{proof}
\subsection{Stochastic Approximation of Projections onto Smooth Sets}
 Let $C \subset \mathbb{R}^n$ be a nonempty closed set with nonempty interior. Recall that
$$
 p_{\delta}(x):=\frac{\mathbb{E}_{y\sim \mathcal{N}(x,\delta  I)}[y\cdot  \mathbf{1}_C(y)]}{\mathbb{E}_{y\sim \mathcal{N}(x,\delta  I)}[\mathbf{1}_C(y)]}=\mathbb{E}_{y\sim \mathcal{N}(x,\delta I)}[y\mid y\in C]. 
$$
\begin{definition}
Let $C\subset \mathbb{R}^n$ be a closed convex set and let $p\in \partial C$. We say that $\partial C$
admits a \emph{local $C^{2,1}$ chart} at $p$ with constants $(\rho,L,M)$ (see Figure~\ref{fig:convexo}) if, after a rigid motion
sending $p$ to $0$ and  the outward unit normal at $p$ to $e_n$, there exists $\rho>0$ and a concave function
$h\in C^{2,1}(\mathbb{B}^{n-1}_{\rho})$ such that $h(0)=0$, $\nabla h(0)=0$, $\Vert \nabla^2 h(z)\Vert \leq L$, 
\[
C\cap \big(\mathbb{B}^{n-1}_{\rho}\times (-\rho,\rho)\big)
=\{(z,t)\in \mathbb{R}^{n-1}\times \mathbb{R}:\ t\le h(z)\},
\]
and $\nabla^2 h$ is Lipschitz on $\mathbb{B}^{n-1}_{\rho}$ with constant $M\ge 0$, i.e.,
\[
\|\nabla^2 h(z)-\nabla^2 h(z')\|\le M\|z-z'\|\qquad \text{for all } z,z'\in \mathbb{B}^{n-1}_{\rho}.
\]
\begin{figure}[h]
  \centering
    \includegraphics{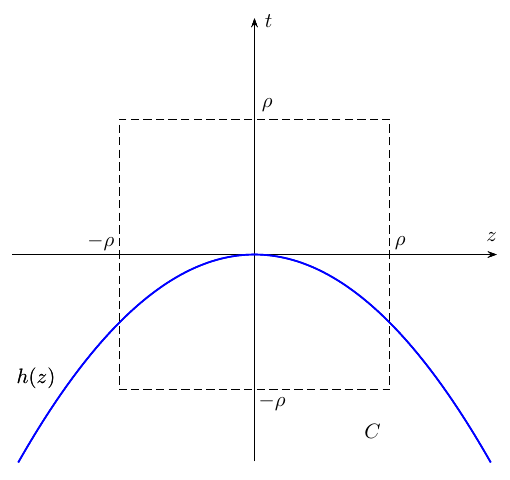}
\caption{Local $C^{2,1}$ chart of $\partial C$.}
      \label{fig:convexo}
\end{figure} 
\end{definition}
In the case of a smooth convex set, we obtain the following refinement of Corollary~\ref{cor:convex-set}, which yields a better convergence rate of order $\delta$. 

\begin{theorem}\label{prop-C21}
Let $C \subset \mathbb{R}^n$ be a closed convex set with nonempty interior, and fix
$x\in \mathbb{R}^n\setminus C$. Assume that $\partial C$ admits a local $C^{2,1}$ chart at
$\operatorname{proj}_C(x)$ with constants $(\rho,L,M)$. Then there exists $\delta_0>0$ such that for every $0<\delta<\delta_0$, we have
\begin{equation*}
\begin{aligned}
\|p_{\delta}(x)-\operatorname{proj}_C(x)\|
= &\mathcal{O}(\delta). 
\end{aligned}
\end{equation*}
\end{theorem}
\begin{proof}
The proof can be found in Section \ref{proof43}.
\end{proof} 
\begin{remark} A careful review of the proof of Theorem~\ref{prop-C21} shows that, for this result, it suffices that the set $C$ be locally a convex body in a neighborhood of the (necessarily unique) projection point.
\end{remark}

The next example shows that the order of the bound obtained in Theorem \ref{prop-C21} is sharp.
\begin{example} Consider the half-space $C=\{y\in \mathbb{R}^n: \langle y, \nu\rangle \leq 0\}$, where $\Vert \nu\Vert=1$. Then $\partial C=\{y\in\mathbb{R}^n: \langle y,\nu\rangle=0\}$ is a hyperplane. Pick $x\notin C$, so that $\langle x,\nu\rangle>0$. Then
$$
\operatorname{proj}_C(x)=x-\langle x,\nu \rangle \nu, \quad d_C(x)=\langle x,\nu\rangle.
$$
Moreover, since $p_{\delta}(x)=\mathbb{E}_{y\sim \mathcal{N}(x,\delta I)}\!\left[\,y \mid y\in C\,\right]$, symmetry implies that the conditional mean in tangential directions is unchanged and the drift is purely normal. Hence, 
$$
p_{\delta}(x)-\operatorname{proj}_C(x)=\alpha_{\delta}\,\nu
$$
for some scalar $\alpha_{\delta}<0$. In fact, one can show that
$$
\alpha_{\delta}:=d_C(x)-\sqrt{\delta}\frac{\phi(d_C(x)/\sqrt{\delta})}{\Phi(-d_C(x)/\sqrt{\delta})},
$$
where $\phi$ and $\Phi$ denote the standard normal density and distribution function, respectively.
Therefore
$$
p_{\delta}(x)=\operatorname{proj}_C(x)+\left(d_C(x)-\sqrt{\delta}\frac{\phi(d_C(x)/\sqrt{\delta})}{\Phi(-d_C(x)/\sqrt{\delta})}\right)\nu,
$$
which implies that $\|p_{\delta}(x)-\operatorname{proj}_C(x)\|=O(\delta)$  as $\delta \to 0^{+}$. Hence, the  order of the bound obtained in Theorem \ref{prop-C21} is sharp.  
\end{example}

\paragraph{Acknowledgements}

D. Morales was supported by ANID Chile under the CMM BASAL funds for the Center of Excellence FB210005 and ANID BECAS/DOCTORADO NACIONAL 21252113. P. P\'erez-Aros was supported by ANID (Chile) through Fondecyt Regular grants No.~1220886, No.~1240120, and No.~1261728; CMM BASAL funds for the Center of Excellence FB210005; and the projects ECOS230027 and MATH-AmSud 23-MATH-17. E. Vilches was supported by ANID (Chile) through Fondecyt Regular grants No.~1220886, No.~1240120, and No.~1261728; CMM BASAL funds for the Center of Excellence FB210005; and the projects ECOS230027 and MATH-AmSud 23-MATH-17.

\bibliographystyle{abbrv}
\bibliography{bib.bib}

@article {MR4581306,
    AUTHOR = {Osher, Stanley and Heaton, Howard and Fung, Samy Wu},
     TITLE = {A {H}amilton-{J}acobi-based proximal operator},
   JOURNAL = {Proc. Natl. Acad. Sci. USA},
  FJOURNAL = {Proceedings of the National Academy of Sciences of the United
              States of America},
    VOLUME = {120},
      YEAR = {2023},
    NUMBER = {14},
     PAGES = {},
      ISSN = {0027-8424,1091-6490},
   MRCLASS = {90C26 (49L25)},
  MRNUMBER = {4581306},
       DOI = {10.1073/pnas.2220469120},
       URL = {https://doi.org/10.1073/pnas.2220469120},
}

@article {MR4923371,
    AUTHOR = {Tibshirani, Ryan J. and Fung, Samy Wu and Heaton, Howard and
              Osher, Stanley},
     TITLE = {Laplace meets {M}oreau: smooth approximation to infimal
              convolutions using {L}aplace's method},
   JOURNAL = {J. Mach. Learn. Res.},
  FJOURNAL = {Journal of Machine Learning Research (JMLR)},
    VOLUME = {26},
      YEAR = {2025},
     PAGES = {},
      ISSN = {1532-4435,1533-7928},
   MRCLASS = {49J53 (35D40 35F21 49L12 90C25)},
  MRNUMBER = {4923371},
}

@article {MR450480,
    AUTHOR = {Brascamp, Herm Jan and Lieb, Elliott H.},
     TITLE = {On extensions of the {B}runn-{M}inkowski and
              {P}r\'{e}kopa-{L}eindler theorems, including inequalities for
              log concave functions, and with an application to the
              diffusion equation},
   JOURNAL = {J. Functional Analysis},
  FJOURNAL = {Journal of Functional Analysis},
    VOLUME = {22},
      YEAR = {1976},
    NUMBER = {4},
     PAGES = {366--389},
      ISSN = {0022-1236},
   MRCLASS = {26A87 (26A51)},
  MRNUMBER = {450480},
MRREVIEWER = {L.\ Leindler},
       DOI = {10.1016/0022-1236(76)90004-5},
       URL = {https://doi.org/10.1016/0022-1236(76)90004-5},
}

@article {MR5558,
    AUTHOR = {Gordon, Robert D.},
     TITLE = {Values of {M}ills' ratio of area to bounding ordinate and of
              the normal probability integral for large values of the
              argument},
   JOURNAL = {Ann. Math. Statistics},
  FJOURNAL = {Annals of Mathematical Statistics},
    VOLUME = {12},
      YEAR = {1941},
     PAGES = {364--366},
      ISSN = {0003-4851},
   MRCLASS = {62.0X},
  MRNUMBER = {5558},
MRREVIEWER = {Z.\ W.\ Birnbaum},
       DOI = {10.1214/aoms/1177731721},
       URL = {https://doi.org/10.1214/aoms/1177731721},
}

@book {MR3616647,
    AUTHOR = {Bauschke, H.H. and Combettes, P.L.},
     TITLE = {Convex analysis and monotone operator theory in {H}ilbert
              spaces},
    SERIES = {{CMS} Books Math./Ouvrages Math. {SMC}},
   EDITION = {2nd},
  publisher = {Springer},
  address   = {Cham},
      YEAR = {2017},
     PAGES = {xix+619},
      ISBN = {978-3-319-48310-8; 978-3-319-48311-5},
       DOI = {10.1007/978-3-319-48311-5},
       URL = {https://doi.org/10.1007/978-3-319-48311-5}
}

@book {MR3288271,
    AUTHOR = {Attouch, Hedy and Buttazzo, Giuseppe and Michaille,
              G\'{e}rard},
     TITLE = {Variational analysis in {S}obolev and {BV} spaces},
    SERIES = {MOS-SIAM Series on Optimization},
    VOLUME = {17},
   EDITION = {2nd},
PUBLISHER = {SIAM},
ADDRESS   = {Philadelphia, PA},
      YEAR = {2014},
     PAGES = {xii+793},
      ISBN = {978-1-611973-47-1},
   MRCLASS = {49-02 (46E35 49J45 49J53 49K20 74G65)},
  MRNUMBER = {3288271},
MRREVIEWER = {Luca\ Granieri},
       DOI = {10.1137/1.9781611973488},
       URL = {https://doi.org/10.1137/1.9781611973488},
}

@book {MR3409135,
    AUTHOR = {Evans, Lawrence C. and Gariepy, Ronald F.},
     TITLE = {Measure theory and fine properties of functions},
    SERIES = {Textb. Math.},
   EDITION = {Revised},
 publisher = {CRC Press},
  address   = {Boca Raton, FL},
      YEAR = {2015},
     PAGES = {xiv+299},
      ISBN = {978-1-4822-4238-6},
   MRCLASS = {28-01},
  MRNUMBER = {3409135},
}

@book {MR257325,
    AUTHOR = {Federer, Herbert},
     TITLE = {Geometric measure theory},
    SERIES = {Grundlehren Math. Wiss.},
    VOLUME = {Band 153},
  PUBLISHER = {Springer-Verlag},
  ADDRESS   = {New York},
      YEAR = {1969},
     PAGES = {xiv+676},
   MRCLASS = {28.80 (26.00)},
  MRNUMBER = {257325},
MRREVIEWER = {J.\ E.\ Brothers},
}

@book {MR775683,
    AUTHOR = {Grisvard, P.},
     TITLE = {Elliptic problems in nonsmooth domains},
    SERIES = {Monographs and Studies in Mathematics},
    VOLUME = {24},
  PUBLISHER = {Pitman},
  ADDRESS   = {Boston, MA},
      YEAR = {1985},
     PAGES = {xiv+410},
      ISBN = {0-273-08647-2},
   MRCLASS = {35J25 (35-02)},
  MRNUMBER = {775683},
MRREVIEWER = {P.\ Szeptycki},
}

@book {MR3726909,
    AUTHOR = {Leoni, Giovanni},
     TITLE = {A first course in {S}obolev spaces},
    SERIES = {Grad. Stud. Math.},
    VOLUME = {181},
   EDITION = {2nd},
 publisher = {American Mathematical Society},
  address   = {Providence, RI},
      YEAR = {2017},
     PAGES = {xxii+734},
      ISBN = {978-1-4704-2921-8},
   MRCLASS = {46E35 (26Axx 26B30 28A78 46-01)},
  MRNUMBER = {3726909},
       DOI = {10.1090/gsm/181},
       URL = {https://doi.org/10.1090/gsm/181},
}

@article {MR201952,
    AUTHOR = {Moreau, Jean-Jacques},
     TITLE = {Proximit\'{e} et dualit\'{e} dans un espace hilbertien},
   JOURNAL = {Bull. Soc. Math. France},
  FJOURNAL = {Bulletin de la Soci\'{e}t\'{e} Math\'{e}matique de France},
    VOLUME = {93},
      YEAR = {1965},
     PAGES = {273--299},
      ISSN = {0037-9484},
   MRCLASS = {46.15 (41.60)},
  MRNUMBER = {201952},
MRREVIEWER = {I.\ G.\ Amemiya},
       URL = {http://www.numdam.org/item?id=BSMF_1965__93__273_0},
}


\appendix
\section{Auxiliary lemmas}

\begin{lemma}\label{co-area}
Let $D\subset \mathbb{R}^n$ be a closed convex set with nonempty interior, and for $R>0$ set $\Omega_R:=\operatorname{int}(D)\cap B_R$. Then $\Omega_R$ is a bounded Lipschitz domain. Moreover, for a.e. $R>0$,  
\begin{equation}\label{eq:corner-null}
\mathcal{H}^{n-1}\big(\partial D\cap \partial B_R\big)=0.    
\end{equation}
Let $F\in W^{1,1}_{\operatorname{loc}}(\operatorname{int}(D);\mathbb{R}^n)$. Then, for every $R>0$ such that \eqref{eq:corner-null} holds and $\int_{\Omega_{R}}(\vert F(x)\vert +\Vert \nabla F(x)\Vert)\,dx<\infty$, we have
\begin{equation*}
\int_{\Omega_R}\operatorname{div}F(x)\,dx
=
\int_{\partial D\cap B_R}F\cdot \nu_D\,d\mathcal{H}^{n-1}
+
\int_{\operatorname{int}(D)\cap \partial B_R} F\cdot \nu_{B_R}\,d\mathcal{H}^{n-1},
\end{equation*}
where $\nu_D$ is the outer unit normal to $D$, defined $\mathcal{H}^{n-1}$-a.e. on $\partial D$,
$\nu_{B_R}$ is the outer unit normal to $B_R$. Here, $F$ on $\partial\Omega_R$ denotes the \emph{interior} Sobolev trace of $F|_{\Omega_R}$ on $\partial\Omega_R$,
restricted to $\partial D\cap B_R$ and to $\operatorname{int}(D)\cap \partial B_R$.
\end{lemma}
\begin{proof}
The first assertion follows from \cite[Chapter~1]{MR775683}, since every convex domain has a Lipschitz boundary. To prove \eqref{eq:corner-null}, we observe that, since $D$ is convex, its boundary $\partial D$ is countably $(n-1)$-rectifiable. Hence, there exist countably many Lipschitz maps $\Phi_k\colon U_k \subset \mathbb{R}^{n-1}\to \mathbb{R}^n$ such that
$$
\mathcal{H}^{n-1}\left(\partial D \setminus \bigcup_{k=1}^\infty \Phi_k(U_k) \right)=0.
$$
see, e.g., \cite{MR257325}. Fix $k$. Define the Lipschitz function $g_k\colon U_k \to \mathbb{R}$ by $g_k(u):=\Vert \Phi_k(u)\Vert$. By the coarea formula for Lipschitz functions (see \cite[Chapter~1]{MR3409135}), for a.e. $R>0$ the level set $g_k^{-1}(R)$ is $(n-2)$-rectifiable and has locally finite $\mathcal{H}^{n-2}$-measure. In particular, $\mathcal{L}^{n-1}(g_k^{-1}(R))=0$. Then, using the area estimate for Lipschitz maps (a consequence of the area formula),
$$
\mathcal{H}^{n-1}(\Phi_k(g_k^{-1}(R)))\leq \operatorname{Lip}(\Phi_k)^{n-1} \mathcal{L}^{n-1}(g_k^{-1}(R))=0.
$$
Moreover, $\partial D\cap  \Phi_k(g_k^{-1}(R))\subset \partial D\cap \partial B_R$ and $\partial D\cap \partial {B}_{R'} \cap \Phi_k(g_k^{-1}(R)) = \emptyset$ for all $R' \neq R$. Hence,  $$(\partial D\cap \partial B_R)\cap \Phi_k(U_k) = (\partial D\cap \partial B_R)\cap \Phi_k(g^{-1}_k(R)), $$ which justifies that 
$$
\mathcal{H}^{n-1}\big((\partial D\cap \partial B_R)\cap \Phi_k(U_k)\big)=0.
$$
Taking the countable union over $k$ and using the $\mathcal{H}^{n-1}$-negligibility of the uncovered part of $\partial D$, we obtain \eqref{eq:corner-null}. To prove the last assertion, since $\Omega_R\subset \operatorname{int}(D)$ is open and bounded, the assumptions on $F$ imply that $F\in W^{1,1}(\Omega_R;\mathbb{R}^n)$ and $\operatorname{div}F\in L^1(\Omega_R)$. Since $\Omega_R$ is a bounded Lipschitz domain, the (interior) trace operator $\operatorname{Tr}\colon W^{1,1}(\Omega_R)\to L^1(\partial \Omega_R)$ is well-defined, and the Gauss-Green formula for Sobolev vector fields (see, e.g., \cite[Chapter~6]{MR3726909}) yields
\begin{equation}\label{Gauss-Green}
\int_{\Omega_R}\operatorname{div}F\,dx=\int_{\partial \Omega_R}\operatorname{Tr}(F)\cdot \nu_{\Omega_R}\, d\mathcal{H}^{n-1}.    
\end{equation}
Next, we observe that
$$
\partial \Omega_R=(\partial D\cap B_R)\cup (\operatorname{int}(D)\cap \partial B_R)\cup (\partial D \cap \partial B_R).
$$
Moreover, on $\partial D\cap B_R$ the boundary $\partial \Omega_R$ coincides locally with $\partial D$.  Hence $\nu_{\Omega_R}=\nu_D$ $\mathcal{H}^{n-1}$-a.e. on $\partial D\cap B_R$. Finally, by \eqref{eq:corner-null}, the set $\partial D\cap \partial B_R$ has $\mathcal{H}^{n-1}$-measure zero and does not contribute to the boundary integral in \eqref{Gauss-Green}. Therefore \eqref{Gauss-Green} splits into
$$
\int_{\Omega_R}\operatorname{div}F\,dx=\int_{\partial D\cap B_R}\operatorname{Tr}(F)\cdot \nu_D\, d\mathcal{H}^{n-1}+\int_{\operatorname{int}(D)\cap \partial B_R}\operatorname{Tr}(F)\cdot \nu_{B_R}\,d\mathcal{H}^{n-1},
$$
as claimed.
\end{proof}

The following lemma can be proved directly using polar coordinates and integration by parts.
\begin{lemma}\label{Lemma-integral2}
Let $n\ge 1$, $R>0$, $d>0$, and $k\in \{0,1\}$. Then, for every $0<\delta\le \frac{R^2}{2n}$, one has
\[
\int_{\{y\in\R^n:\ \|y\|\ge R\}} (\Vert y\Vert +d)^k\exp\!\Big(-\frac{\|y\|^2}{2\delta}\Big)\,dy
\le
\frac{4\pi^{n/2}}{\Gamma\!\left(\frac n2\right)} (R+d)^k R^{n-2}\delta\,
\exp\!\Big(-\frac{R^2}{2\delta}\Big).
\]
\end{lemma}
The following lemma can be found in \cite{MR5558}.
\begin{lemma}\label{cota-inf-normal} For all $\delta>0$ and $y>0$,
$$
\int_{y}^{\infty}e^{-\frac{t^2}{2\delta}}\,dt\geq \frac{\delta y}{y^2+\delta}e^{-\frac{y^2}{2\delta}}.
$$ 
\end{lemma}

\section{Proof of Theorem \ref{prop-C21}}\label{proof43}

\textbf{Step 1: Initialization}\\
Set $p:=\operatorname{proj}_C(x)$, $d:=d_{C}(x)$, and $\nu:=\frac{x-p}{d}$. After translating
and rotating, we may assume that $p=0$, $\nu=e_n$, and $x=de_n$. For $y\in \mathbb{R}^n$,
write $y=(z,t)\in \mathbb{R}^{n-1}\times \mathbb{R}$. Let $H=\nabla^2 h(0)$ (hence $H\preceq 0$ by concavity, and $\Vert H\Vert \leq L$). By Taylor's theorem with Lipschitz Hessian, for all $z\in \mathbb{B}_{\rho}^{n-1}$.
\begin{equation}\label{Taylor3}
h(z)=\frac{1}{2}\langle Hz,z\rangle+r(z), \quad \vert r(z)\vert \leq \frac{M}{6}\Vert z\Vert^3,
\end{equation}
and also
\begin{equation}\label{eq:h_quadratic}
|h(z)|\le \frac{L}{2}\|z\|^2 \textrm{ for all } z\in \mathbb{B}_{\rho}^{n-1}.
\end{equation}
Moreover, by convexity of $C$ and the fact that $e_n$ is an outward normal at $0$ imply that the supporting  hyperplane is  $\{t=0\}$, hence $h(z)\leq 0$ on $\mathbb{B}_{\rho}^{n-1}$.\\
Define (see Figure \ref{fig:convexo-a})
$$
C_{\operatorname{loc}}:=C\cap \big(\mathbb{B}^{n-1}_{\rho}\times (-\rho,\rho)\big) \textrm{ and } C_{\operatorname{tail}}:=C\setminus C_{\operatorname{loc}}.
$$
Let us define
\begin{equation*}
\begin{aligned}
N_{\operatorname{loc}} &:= \int_{C_{\operatorname{loc}}} y e^{-\frac{\| y-x\|^2}{2\delta}} \, dy, &
N_{\operatorname{tail}} &:= \int_{C_{\operatorname{tail}}} y e^{-\frac{\| y-x\|^2}{2\delta}} \, dy, \\
D_{\operatorname{loc}} &:= \int_{C_{\operatorname{loc}}} e^{-\frac{\| y-x\|^2}{2\delta}} \, dy, &
D_{\operatorname{tail}} &:= \int_{C_{\operatorname{tail}}} e^{-\frac{\| y-x\|^2}{2\delta}} \, dy.
\end{aligned}
\end{equation*}
With the above definitions, we observe that 
$$
p_{\delta}(x):=\frac{N_{\operatorname{loc}}+N_{\operatorname{tail}}}{D_{\operatorname{loc}}+D_{\operatorname{tail}}}.
$$

\noindent \textbf{Step 2: Bounds on $D_{\operatorname{tail}}$ and $D_{\operatorname{loc}}$.}\\
We observe that $C_{\operatorname{loc}}\subset \mathbb{B}_{\rho}^{n-1}\times (-\rho,0]$ and if $y\in C_{\operatorname{tail}}$, then either $\Vert z\Vert \geq \rho$ or $t\leq -\rho$. In both cases,  
\begin{equation}\label{tail-inequality}
\Vert y-x\Vert^2=\Vert z\Vert^2+(t-d)^2\geq \rho^2+d^2 \textrm{ for } y\in C_{\operatorname{tail}}.    
\end{equation}
\begin{figure}[h]
  \centering
    \includegraphics{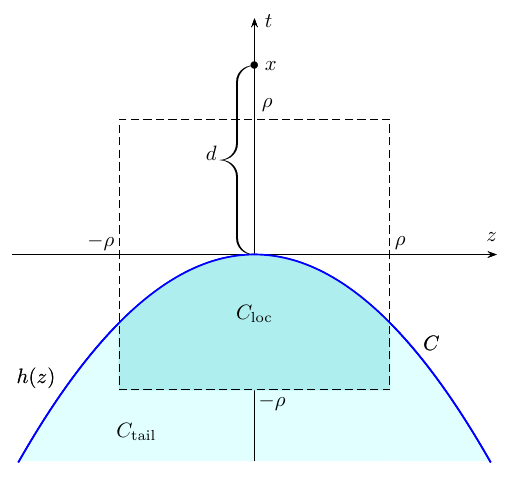}
  \caption{Local chart of $\partial C$ and the sets $C_{\operatorname{loc}}$ (pale turquoise) and $C_{\operatorname{tail}}$ (light cyan).}
      \label{fig:convexo-a}
\end{figure} 

\textbf{Claim 1:} Assume that $0<\delta\leq \frac{\rho^2+d^2}{2n}$ (with the convention that $1/0=+\infty$). Then
$$
D_{\operatorname{tail}}:=\int_{C_{\operatorname{tail}}} e^{-\frac{\Vert y-x\Vert^2}{2\delta}}dy\leq \frac{4\pi^{n/2}}{\Gamma\left(\frac{n}{2}\right)}(\rho^2+d^2)^{\frac{n-2}{2}} \delta e^{-\frac{\rho^2+d^2}{2\delta}}.
$$
\emph{Proof of Claim 1:} Since $C_{\operatorname{tail}}\subset \{ y\in \mathbb{R}^n: \Vert y-x\Vert^2 \geq \rho^2+d^2\}$, we have
$$
D_{\operatorname{tail}}=\int_{C_{\operatorname{tail}}} e^{-\frac{\Vert y-x\Vert^2}{2\delta}}dy\leq \int_{\{ y\in \mathbb{R}^n: \Vert y-x\Vert^2 \geq \rho^2+d^2\}} e^{-\frac{\Vert y-x\Vert^2}{2\delta}}dy.
$$
Applying Lemma~\ref{Lemma-integral2} with $R=\sqrt{\rho^2+d^2}$ yields the claim. \qed \newline
\noindent \textbf{Claim 2:} Assume that $0<\delta \leq  \min \left\{ \rho^2(1+dL),  \frac{2\rho(1+dL)}{L},  \frac{d\rho}{1 + \frac{dL}{2(1+dL)}} \right\}$. \\
Then
$$
D_{\operatorname{loc}}:=\int_{C_{\operatorname{loc}}} e^{-\frac{\Vert y-x\Vert^2}{2\delta}}dy
\geq \left[\psi(\delta) \cdot \mathcal{K}_{\operatorname{err}}(\delta)\cdot \frac{C_{\operatorname{vol}}(n)}{(1+dL)^{\frac{n-1}{2}}} \right]\cdot \delta^{\frac{n-1}{2}}\cdot e^{-\frac{d^2}{2\delta}},
$$
where 
\begin{equation*}
    \begin{aligned}
  \psi(\delta)&:=\inf_{z\in \mathbb{B}^{n-1}_{r_{\delta}}}\omega(z), \quad r_{\delta}:=\sqrt{\frac{\delta}{1+dL}},\\
  \omega(z)&:=\frac{\delta (d-h(z))}{(d-h(z))^2+\delta}-\frac{\delta}{\rho+d} \exp\left(-\frac{\left[(\rho+d)^2-(d-h(z))^2\right]}{2\delta}\right),\\
  \mathcal{K}_{\operatorname{err}}(\delta)&:=\exp\left(-\frac{L^2\delta}{8(1+dL)^2}\right),\\
 C_{\operatorname{vol}}(n)&:=\int_{\mathbb{B}^{n-1}}e^{-\frac{\Vert v\Vert^2}{2}}dv.
    \end{aligned}
\end{equation*}
\emph{Proof of Claim 2}: Since $\delta \leq \rho^2 (1+dL)$, we have that $r_{\delta}\leq \rho$. Set $A(z):=d-h(z)$ and observe that
\begin{itemize}
  \item By \eqref{eq:h_quadratic}, we have
  \[
  A(z)=d-h(z)\le d+\frac{L}{2}r_{\delta}^2
  =d+\frac{L\delta}{2(1+dL)}=:A_{\max}.
  \]
  \item Since \(0<\delta<\frac{2\rho(1+dL)}{L}\), it follows that \(A_{\max}<d+\rho\). Consequently,
  \[
  (d+\rho)^2-A(z)^2\ge (d+\rho)^2-A_{\max}^2.
  \]
  \item Since \(0<\delta<\frac{d\rho}{1+\frac{dL}{2(1+dL)}}\), we obtain the lower bound
  \begin{equation*}
      \begin{aligned}
          J(z)&:=\int_{A(z)}^{d+\rho} e^{-\frac{t^2}{2\delta}}\,dt\\
          &=\int_{A(z)}^{+\infty} e^{-\frac{t^2}{2\delta}}\,dt-\int_{\rho+d}^{\infty} e^{-\frac{t^2}{2\delta}}\,dt\\
          &\geq \frac{\delta A(z)}{A(z)^2+\delta}e^{-\frac{A(z)^2}{2\delta}}-\frac{\delta}{\rho+d}e^{-\frac{(\rho+d)^2}{2\delta}}\\
          &=e^{-\frac{A(z)^2}{2\delta}} \omega(z)\\
  &\geq \psi(\delta)\,e^{-\frac{A(z)^2}{2\delta}},
      \end{aligned}
  \end{equation*}
  where we have used Lemma \ref{cota-inf-normal}.
\end{itemize}
Therefore,
\begin{equation*}
    \begin{aligned}
        D_{\operatorname{loc}}&=\int_{z\in \mathbb{B}_{\rho}^{n-1}}e^{-\frac{\Vert z\Vert^2}{2\delta}}\left(\int_{-\rho}^{h(z)}e^{-\frac{(t-d)^2}{2\delta}}\,dt\right)\,dz\\
        &=\int_{z\in \mathbb{B}_{\rho}^{n-1}}e^{-\frac{\Vert z\Vert^2}{2\delta}}J(z)\,dz\\
        &\geq \int_{z\in \mathbb{B}_{r_{\delta}}^{n-1}} \psi(\delta)e^{-\frac{\Vert z\Vert^2+A(z)^2}{2\delta}}\,dz.
    \end{aligned}
\end{equation*}
Define \(E(z):=\|z\|^2 + A(z)^2\). Since \(\|z\|\le r_{\delta}\), it follows that
\[
E(z)\le d^2 + (1+dL)\|z\|^2 + \frac{L^2}{4}\|z\|^4
\le d^2 + (1+dL)\|z\|^2 + \frac{L^2}{4}\left(\frac{\delta}{1+dL}\right)^2.
\]
Hence,
\begin{equation*}
    \begin{aligned}
        D_{\operatorname{loc}}&\geq \psi(\delta)e^{-\frac{L^2\delta }{8 (1+dL)^2}} e^{-\frac{d^2}{2\delta}} \left(\frac{\delta}{1+dL}\right)^{\frac{n-1}{2}}\int_{\mathbb{B}^{n-1}}e^{-\frac{\Vert v\Vert^2}{2}}dv\\
        &=\left[\psi(\delta) \cdot \mathcal{K}_{\operatorname{err}}(\delta)\cdot \frac{C_{\operatorname{vol}}(n)}{(1+dL)^{\frac{n-1}{2}}} \right]\cdot \delta^{\frac{n-1}{2}}\cdot e^{-\frac{d^2}{2\delta}},
    \end{aligned}
\end{equation*}
which proves the claim.\\
\textbf{Claim 3:} If $0<\delta \leq  \delta_0:=\min \big\{ \rho^2(1+dL),  \frac{2\rho(1+dL)}{L},  \frac{d\rho}{1 + \frac{dL}{2(1+dL)}},\frac{\rho^2+d^2}{2n} \big\}$, then $\psi(\delta)>0$ and 
$$
\frac{D_{\operatorname{tail}}}{D_{\operatorname{loc}}}
\leq \left[\frac{\mathcal{K}_{\operatorname{geom}}}{\mathcal{K}_{\operatorname{err}}(\delta)}\right] \left(\frac{\delta}{\psi(\delta)}\right) \delta^{-\frac{(n-1)}{2}}e^{-\frac{\rho^2}{2\delta}},
$$
where $\mathcal{K}_{\operatorname{err}}(\delta)$ and $\psi(\delta)$ are given in Claim 2. Moreover, 
\begin{equation*}
    \begin{aligned}
\mathcal{K}_{\operatorname{geom}}&:=\frac{4\pi^{n/2}}{\Gamma\left(\frac{n}{2}\right) C_{\operatorname{vol}}(n)}(\rho^2+d^2)^{\frac{n-2}{2}}(1+\rho L)^{\frac{n-1}{2}}.
\end{aligned}
\end{equation*}
\emph{Proof of Claim 3}: We  observe that for $0<\delta \leq \delta_0$,
\begin{equation*}
    \begin{aligned}
    \omega(z)&=\frac{\delta A(z)}{A(z)^2+\delta}-\frac{\delta}{\rho+d} \exp\left(-\frac{\left[(\rho+d)^2-A(z)^2\right]}{2\delta}\right)\\
    &\geq \delta \left(\frac{A_{\operatorname{max}}}{A_{\operatorname{max}}^2+\delta}-\frac{1}{e(\rho+d)}\right)>0,
    \end{aligned}
\end{equation*}
which implies that $\psi(\delta)=\inf_{z\in \mathbb{B}^{n-1}_{r_{\delta}}}\omega(z)>0$. Finally, the result follows from Claims 1 and  2. \qed\\
\noindent \textbf{Claim 4}: Let $\delta_1:=\min\left\{ \delta_0,\frac{d^2}{3}, \frac{\rho^2}{2\ln\!\left(4\left(1+\frac{\rho}{d}\right)\right)}\right\}$. Then, for all $\delta \in (0,\delta_1)$, one has 
\[
\psi(\delta)\geq \frac{\delta}{2d}.
\]
\emph{Proof of Claim 4}: From the definition of $\omega(z)=T_1(z)-T_2(z)$, where
\begin{equation*}
    \begin{aligned}
    T_1(z):=\frac{\delta A(z)}{A(z)^2+\delta}\quad \textrm{ and } \quad  T_2(z):=\frac{\delta}{\rho+d} \exp\left(-\frac{\left[(\rho+d)^2-A(z)^2\right]}{2\delta}\right)
    \end{aligned}
\end{equation*}
Hence, on the one hand, if $\delta \leq d^2/3$, then $T_1(z)\geq \frac{3\delta}{4d}$. On the other hand, if $\delta \leq \frac{\rho^2}{2\ln\left(4(1+\frac{\rho}{d})\right)}$, then $T_2(z)\leq \frac{\delta}{4d}$. Finally,
$$
\psi(\delta)=\inf_{z\in \mathbb{B}^{n-1}_{r_{\delta}}} (T_1(z)-T_2(z))\geq \frac{3\delta}{4d}-\frac{\delta}{4d}=\frac{\delta}{2d}.
$$ \qed\\
With respect to the tangential/normal splitting induced by $y=(z,t)$, we write
$N_{\operatorname{loc}}=(N_{\operatorname{loc}}^{\tau},N_{\operatorname{loc}}^{\nu})$ and $N_{\operatorname{tail}}=(N_{\operatorname{tail}}^{\tau},N_{\operatorname{tail}}^{\nu})$. Hence,
\begin{equation*}
    \begin{aligned}
        N_{\operatorname{loc}}^{\tau}&=\int_{\mathbb{B}_{\rho}^{n-1}} ze^{-\frac{\Vert z\Vert^2}{2\delta}} \left(\int_{-\rho}^{h(z)}  e^{-\frac{(t-d)^2}{2\delta}}\,dt\right)\,dz\\
        N_{\operatorname{loc}}^{\nu}&=\int_{\mathbb{B}_{\rho}^{n-1}} e^{-\frac{\Vert z\Vert^2}{2\delta}} \underbrace{\left(\int_{-\rho}^{h(z)} t\, e^{-\frac{(t-d)^2}{2\delta}}\,dt\right)}_{=:I_{\nu}(z)}\,dz
    \end{aligned}
\end{equation*}
Moreover, we write $p_{\delta}(x)=(p_{\delta}^{\tau}(x),p_{\delta}^{\nu}(x))$, where 
\begin{equation}\label{decomposition-tn}
p_{\delta}^{\tau}(x)=\frac{N_{\operatorname{loc}}^{\tau}+N_{\operatorname{tail}}^{\tau}}{D_{\operatorname{loc}}+D_{\operatorname{tail}}} \quad p_{\delta}^{\nu}(x)=\frac{N_{\operatorname{loc}}^{\nu}+N_{\operatorname{tail}}^{\nu}}{D_{\operatorname{loc}}+D_{\operatorname{tail}}}.
\end{equation}
\noindent\textbf{Step 3: Bounding $N_{\operatorname{tail}}$.}\\
\textbf{Claim 5:} For all $0<\delta\leq \frac{d^2+\rho^2}{2n}$, one has 
$$
\Vert N_{\operatorname{tail}}\Vert\leq \frac{4\pi^{n/2}}{\Gamma\left(\frac{n}{2}\right)}\sqrt{\rho^2+d^2}(\rho^2+d^2)^{\frac{n-2}{2}} \delta e^{-\frac{\rho^2+d^2}{2\delta}}.
$$
\emph{Proof of Claim 5:} From \eqref{tail-inequality}, one has
\begin{equation*}
    \begin{aligned}
\Vert N_{\operatorname{tail}}\Vert &=\left\Vert \int_{C_{\operatorname{tail}}} ye^{-\frac{\Vert y-x\Vert^2}{2\delta}}\,dy\right\Vert \\
&\leq \int_{C_{\operatorname{tail}}} \Vert y\Vert e^{-\frac{\Vert y-x\Vert^2}{2\delta}}\,dy\\
&\leq \int_{\{ u \in \mathbb{R}^n : \Vert u\Vert \geq \sqrt{\rho^2+d^2}\}} (\Vert u\Vert +d)e^{-\frac{\Vert u \Vert^2}{2\delta}}du,
    \end{aligned}
\end{equation*}
which, by virtue of Lemma \ref{Lemma-integral2}, implies the claim.\\
\noindent\textbf{Step 4: Bounding the normal component $N_{\operatorname{loc}}^{\nu}$ of $N_{\operatorname{loc}}$.}\\
\textbf{Claim 6:}  The following formula holds:
\begin{equation}\label{eq-N}
N_{\operatorname{loc}}^{\nu}
= d\,D_{\operatorname{loc}}
-\delta I_{\operatorname{num}}
+\delta\, e^{-\frac{(\rho+d)^2}{2\delta}}\int_{\mathbb{B}_{\rho}^{n-1}} e^{-\frac{\|z\|^2}{2\delta}}\,dz,
\end{equation}
where $I_{\operatorname{num}}:=\int_{\mathbb{B}_{\rho}^{n-1}} e^{-\frac{\|z\|^2 + A(z)^2}{2\delta}}\,dz$.\\
\emph{Proof of Claim 6:} Recall that $N_{\operatorname{loc}}^{\nu}=\int_{\mathbb{B}_{\rho}^{n-1}} e^{-\frac{\|z\|^2}{2\delta}} I_{\nu}(z)\,dz$ with 
\begin{equation*}
    I_{\nu}(z)=\int_{-\rho}^{h(z)} t\, e^{-\frac{(t-d)^2}{2\delta}}\,dt=\underbrace{\int_{-(\rho+d)}^{-A(z)}u e^{-\frac{u^2}{2\delta}}\,du}_{\operatorname{Term}_1(z)}+d\underbrace{\int_{-(\rho+d)}^{-A(z)} e^{-\frac{u^2}{2\delta}}\,du}_{\operatorname{Term}_2(z)}.
\end{equation*}
Moreover, $\operatorname{Term}_1(z)=-\delta \big(e^{-\frac{A(z)^2}{2\delta}}-e^{-\frac{(\rho+d)^2}{2\delta}}\big)$ and $\operatorname{Term}_2(z)=J(z)$. Then, 
$$
 I_{\nu}(z)=-\delta e^{-\frac{A(z)^2}{2\delta}}+\delta e^{-\frac{(\rho+d)^2}{2\delta}}+dJ(z),
$$
which implies that
\begin{equation*}
    \begin{aligned}
N_{\operatorname{loc}}^{\nu}
&=\int_{\mathbb{B}_{\rho}^{n-1}} e^{-\frac{\|z\|^2}{2\delta}}
\left[d\,J(z)-\delta e^{-\frac{A(z)^2}{2\delta}}+\delta e^{-\frac{(\rho+d)^2}{2\delta}}\right]\,dz\\
 &=d\,D_{\operatorname{loc}}-\delta  I_{\operatorname{num}}+\delta\,e^{-\frac{(\rho+d)^2}{2\delta}}\int_{\mathbb{B}_{\rho}^{n-1}} e^{-\frac{\|z\|^2}{2\delta}}\,dz,
    \end{aligned}
\end{equation*}
which proves the claim. \qed\\
We observe that the third term in the right-hand side of equality \eqref{eq-N} is $\mathcal{O}(e^{-c/\delta})$,  $D_{\operatorname{tail}}=\mathcal{O}(\delta e^{-c/\delta})$ (see Claim~1), and $\Vert N_{\operatorname{tail}}\Vert=\mathcal{O}(\delta e^{-c/\delta})$ (see Claim~5). Hence, informally speaking, by  Claim 6, we get
$$
p_{\delta}^{\nu}(x)=\frac{N_{\operatorname{loc}}^{\nu}+N_{\operatorname{tail}}^{\nu}}{D_{\operatorname{loc}}+D_{\operatorname{tail}}} \approx \frac{N_{\operatorname{loc}}^{\nu}}{D_{\operatorname{loc}}}\approx  d-\delta \frac{I_{\operatorname{num}}}{D_{\operatorname{loc}}}.
$$
Therefore, we consider the estimator
\begin{equation*}
    \tilde{p}_{\delta}^{\nu}:=d-\delta \frac{I_{\operatorname{num}}}{D_{\operatorname{loc}}}.
\end{equation*}
Let us consider the measure $\eta_{\delta}$ defined on $\mathbb{B}^{n-1}_{\rho}$ whose density is proportional to 
$$e^{-\frac{\Vert z\Vert^2+(d-h(z))^2}{2\delta}}.$$
\textbf{Claim 7:}  Assume that $0<\delta\leq \min \left\{ \rho^2(1+dL),  \frac{2\rho(1+dL)}{L},  \frac{d\rho}{1 + \frac{dL}{2(1+dL)}} \right\}$. Then there exists a function $r(A(z),\delta)$ such that $0<r(A(z),\delta)<\frac{\delta}{A(z)^2}$ and
\begin{equation*}
    \frac{I_{\operatorname{num}}}{D_{\operatorname{loc}}}=\frac{1}{\delta \mathbb{E}_{\eta_{\delta}}\left[\frac{1-r(A(z),\delta)}{A(z)}\right]-\Delta(\delta)} \quad \textrm{ for all } 0<\delta \leq \delta_1,
\end{equation*}
where
$$
0\leq \Delta(\delta)\leq \frac{(1+dL)^{\frac{n-1}{2}} }{\mathcal{K}_{\operatorname{err}} \cdot (\rho +d)} \rho^{n-1}  \frac{\operatorname{Vol}(\mathbb{B}^{n-1})}{\int_{\mathbb{B}^{n-1}} e^{-\frac{\Vert v\Vert^2}{2}}dv} \cdot  \delta^{\frac{3-n}{2}}e^{-\frac{\rho^2+2d\rho}{2\delta}}=\mathcal{O}( \delta^{\frac{3-n}{2}} e^{-c/\delta}).
$$
\emph{Proof of Claim 7:} Let us observe that 
\begin{equation*}
\begin{aligned}
J(z)
&:=\int_{A(z)}^{d+\rho} e^{-\frac{t^2}{2\delta}}\,dt \\
&=\int_{A(z)}^{\infty} e^{-\frac{t^2}{2\delta}}\,dt-\int_{d+\rho}^{\infty} e^{-\frac{t^2}{2\delta}}\,dt \\
&=\frac{\delta}{A(z)}e^{-\frac{A(z)^2}{2\delta}}
\big(1-\underbrace{A(z)e^{\frac{A(z)^2}{2\delta}}\int_{A(z)}^{\infty} \frac{1}{t^2}e^{-\frac{t^2}{2\delta}}\,dt}_{=:r(A(z),\delta)}\big)-\int_{d+\rho}^{\infty} e^{-\frac{t^2}{2\delta}}\,dt.
\end{aligned}
\end{equation*}
Moreover, since $t\ge A(z)$ on $[A(z),\infty)$, we have $1/t^2 \le 1/A(z)^2$. Using this bound and the
standard Gaussian tail estimate (obtained by integration by parts), for all $w>0$
\begin{equation}\label{Mills}
\int_{w}^{\infty} e^{-\frac{t^2}{2\delta}}\,dt \le \frac{\delta}{w}e^{-\frac{w^2}{2\delta}},    
\end{equation}
we obtain
\begin{equation*}
\begin{aligned}
0<r(A(z),\delta)&=A(z)e^{\frac{A(z)^2}{2\delta}}\int_{A(z)}^{\infty} \frac{1}{t^2}e^{-\frac{t^2}{2\delta}}\,dt\\
&\leq A(z)e^{\frac{A(z)^2}{2\delta}}\cdot \frac{1}{A(z)^2}\int_{A(z)}^{\infty} e^{-\frac{t^2}{2\delta}}\,dt\\
&\leq A(z)e^{\frac{A(z)^2}{2\delta}}\cdot \frac{1}{A(z)^2}\left(\frac{\delta}{A(z)}e^{-\frac{A(z)^2}{2\delta}}\right)\\
&=\frac{\delta}{A(z)^2}.
\end{aligned}
\end{equation*}
Therefore,
\begin{equation*}
    \begin{aligned}
 \frac{I_{\operatorname{num}}}{D_{\operatorname{loc}}}&=\frac{\int_{\mathbb{B}^{n-1}_{\rho}} e^{-\frac{\Vert z\Vert^2+A(z)^2}{2\delta}}\,dz }{\int_{\mathbb{B}^{n-1}_{\rho}} e^{-\frac{\Vert z\Vert^2}{2\delta}}J(z)\,dz}\\
 &=\frac{1}{\delta \mathbb{E}_{\eta_{\delta}}\left[\frac{1-r(A(z),\delta)}{A(z)}\right]-\Delta(\delta)},
     \end{aligned}
\end{equation*}
where
$$
\Delta(\delta):=\frac{\int_{\rho+d}^{\infty} e^{-\frac{u^2}{2\delta}}\,du \cdot \int_{\mathbb{B}^{n-1}_{\rho}} e^{-\frac{\Vert z\Vert^2}{2\delta}}\,dz}{\int_{\mathbb{B}^{n-1}_{\rho}} e^{-\frac{\Vert z\Vert^2+A(z)^2}{2\delta}}\,dz}.
$$
On the one hand, by inequality \eqref{Mills}, we observe that the numerator of $\Delta(\delta)$ is bounded by
$$
\frac{\delta}{\rho+d}e^{-\frac{(\rho+d)^2}{2\delta}} \operatorname{Vol}(\mathbb{B}^{n-1}_{\rho}).
$$
On the other hand, the analysis for the denominator is similar to the analysis of $D_{\operatorname{loc}}$ (see Claim 2). Hence,
$$
\int_{\mathbb{B}^{n-1}_{\rho}} e^{-\frac{\Vert z\Vert^2+A(z)^2}{2\delta}}\,dz\geq \mathcal{K}_{\operatorname{err}}\cdot e^{-\frac{d^2}{2\delta}} \left(\frac{\delta}{1+dL}\right)^{\frac{n-1}{2}}  \int_{\mathbb{B}^{n-1}}e^{-\frac{\Vert v\Vert^2}{2}}dv.
$$
Combining the above inequalities, we get the claim. \qed\\
\textbf{Claim 8:} We have
\[
\tilde{p}_{\delta}^{\nu}
=
\mathbb{E}_{\eta_{\delta}}[h(z)]
- d\,\mathbb{E}_{\eta_{\delta}}\!\left[r(A(z),\delta)\right] 
+ O(\delta^{2}) \textrm{ as } \delta \to 0^+.
\]
Moreover,  
$$
\vert \tilde{p}_{\delta}^{\nu}\vert \leq \left(\frac{L(n-1)}{2}+\frac{1}{d}\right)\delta +O(\delta^{3/2}) \textrm{ as } \delta \to 0^+.
$$
\emph{Proof of Claim 8:} Let $g_{\delta}(z)=\frac{1-r(A(z),\delta)}{A(z)}$. Then,  by virtue of Claim 7, we get
\begin{equation*}
    \begin{aligned}
      \tilde{p}_{\delta}^{\nu}&=d-\delta \frac{I_{\operatorname{num}}}{D_{\operatorname{loc}}}
      =d- \frac{\delta }{\delta \mathbb{E}_{\eta_{\delta}}[g_{\delta}(z)]}+O(\delta^{\frac{1-n}{2}}e^{-c/\delta}).
    \end{aligned}
\end{equation*}
Moreover, by Taylor's theorem,
\begin{equation*}
    \begin{aligned}
    \frac{1}{\mathbb{E}_{\eta_{\delta}}[g_{\delta}(z)]}=\frac{1}{\frac{1}{d}+\mathbb{E}_{\eta_{\delta}}\left[g_{\delta}(z)-\frac{1}{d}\right]}=d-d^2 \mathbb{E}_{\eta_{\delta}}\left[g_{\delta}(z)-\frac{1}{d}\right]+O(\delta^2).
    \end{aligned}
\end{equation*}
Therefore, 
$$
 \tilde{p}_{\delta}^{\nu}= \mathbb{E}_{\eta_{\delta}}[h(z)]-d\,\mathbb{E}_{\eta_{\delta}}[r(A(z),\delta)]+O(\delta^2).
$$
Finally, using that $\vert h(z)\vert \leq \frac{L}{2}\Vert z\Vert^2$ and $\mathbb{E}_{\eta_{\delta}}[\Vert z\Vert^2]\leq (n-1)\delta +O(\delta^{3/2})$, we obtain that 
$$
\vert \tilde{p}_{\delta}^{\nu}\vert \leq \left(\frac{L(n-1)}{2}+\frac{1}{d}\right)\delta+O(\delta^{3/2}),
$$
which proves the claim.\\
\noindent\textbf{Step 4: Bounding the tangential component $N_{\operatorname{loc}}^{\tau}$ of $N_{\operatorname{loc}}$.}\\
\textbf{Claim 9:} For all $0<\delta \leq \delta_1$
$$
\Vert N_{\operatorname{loc}}^{\tau}\Vert \leq \frac{M(n^2-1)}{6} (2\pi \delta)^{\frac{n-1}{2}} \delta^2 e^{-\frac{d^2}{2\delta}}.
$$
\emph{Proof of Claim 9:} Recall that
$$
N_{\operatorname{loc}}^{\tau}= \int_{\mathbb{B}_{\rho}^{n-1}} ze^{-\frac{\Vert z\Vert^2}{2\delta}}J(z)dz.
$$
Since $\mathbb{B}_{\rho}^{n-1}$ is symmetric and $e^{-\frac{\Vert z\Vert^2}{2\delta}}$ is even, 
$$
N_{\operatorname{loc}}^{\tau}=\frac{1}{2}\int_{\mathbb{B}_{\rho}^{n-1}} ze^{-\frac{\Vert z\Vert^2}{2\delta}} (J(z)-J(-z))dz.
$$
Hence
$$
\Vert N_{\operatorname{loc}}^{\tau}\Vert \leq \frac{1}{2}\int_{\mathbb{B}_{\rho}^{n-1}} \Vert z\Vert e^{-\frac{\Vert z\Vert^2}{2\delta}} \vert J(z)-J(-z)\vert dz.
$$
Set $F(u):=\int_{-\rho}^u e^{-\frac{(t-d)^2}{2\delta}}dt$. Hence, $J(z)=F(h(z))$ and $F^{\prime}(u)=e^{-\frac{(u-d)^2}{2\delta}}$.
By the mean value theorem, for each $z$ there exists $\theta_z$ between $h(z)$ and $h(-z)$ such that
$$
J(z)-J(-z)=F^{\prime}(\theta_z)(h(z)-h(-z)).
$$
Moreover, since $h(z)\leq 0$ and $h(-z)\leq 0$, we have $\theta_z\leq 0$, hence $(\theta_z-d)^2\geq d^2$ and thus
$$
\vert F^{\prime}(\theta_z)\vert \leq e^{-\frac{d^2}{2\delta}}.
$$
Therefore,
$$
\vert J(z)-J(-z)\vert\leq e^{-\frac{d^2}{2\delta}} \vert h(z)-h(-z)\vert.
$$
From \eqref{Taylor3}, we get that
$$
h(z)=\frac{1}{2}\langle Hz,z\rangle +r(z), \quad \vert r(z)\vert \leq \frac{M}{6}\Vert z\Vert^3.
$$
Since the quadratic term is even, 
$$
h(z)-h(-z)=r(z)-r(-z),
$$
and hence $\vert J(z)-J(-z)\vert \leq \frac{M}{3}\Vert z\Vert^3 e^{-\frac{d^2}{2\delta}}$. Using the above estimates, we get
$$
\Vert N_{\operatorname{loc}}^{\tau}\Vert \leq \frac{M}{6}e^{-\frac{d^2}{2\delta}}\int_{\mathbb{B}_{\rho}^{n-1}} \Vert z\Vert^4 e^{-\frac{\Vert z\Vert^2}{2\delta}}dz\leq  \frac{M}{6}e^{-\frac{d^2}{2\delta}}\int_{\mathbb{R}^{n-1}} \Vert z\Vert^4 e^{-\frac{\Vert z\Vert^2}{2\delta}}dz
$$
If $W\sim \mathcal{N}(0,\delta I_{n-1})$, then
$$
\int_{\mathbb{R}^{n-1}} \Vert z\Vert^4 e^{-\frac{\Vert z\Vert^2}{2\delta}}dz=(2\pi \delta)^{\frac{n-1}{2}} \mathbb{E}\Vert W\Vert^4, \quad \mathbb{E}\Vert W\Vert^4= \delta^2 (n^2-1),
$$
which yields the claim. \qed\\
\noindent\textbf{Claim 10:} Under the assumptions of Theorem \ref{prop-C21}, we have
$$
\Vert p_{\delta}^{\tau}(x)\Vert \leq \frac{M(n^2-1)d}{6}\delta +O(\delta^2) \textrm{ as } \delta \to 0^{+}.
$$
\emph{Proof of Claim 10:}  By \eqref{decomposition-tn}, 
$$
\Vert p_{\delta}^{\tau}(x)\Vert \leq \frac{\Vert N_{\operatorname{loc}}^{\tau}\Vert }{D_{\operatorname{loc}}+D_{\operatorname{tail}}}+\frac{\Vert N_{\operatorname{tail}}\Vert }{D_{\operatorname{loc}}+D_{\operatorname{tail}}}\leq \frac{\Vert N_{\operatorname{loc}}^{\tau}\Vert }{D_{\operatorname{loc}}}+\frac{\Vert N_{\operatorname{tail}}\Vert }{D_{\operatorname{loc}}}.
$$
Using Claim 9 and Claim 2-4, for $0<\delta \leq \delta_1$ we have
\begin{equation*}
    \begin{aligned}
        \frac{\Vert N_{\operatorname{loc}}^{\tau}\Vert }{D_{\operatorname{loc}}}&\leq \frac{\frac{M(n^2-1)}{6}(2\pi\delta)^{\frac{n-1}{2}}\delta^2 e^{-\frac{d^2}{2\delta}}}
{\psi(\delta)\mathcal K_{\rm err}(\delta)\frac{C_{\rm vol}(n)}{(1+dL)^{\frac{n-1}{2}}}\delta^{\frac{n-1}{2}}e^{-\frac{d^2}{2\delta}}}\\
&=
\frac{M(n^2-1)}{6}\,
\frac{(2\pi)^{\frac{n-1}{2}}(1+dL)^{\frac{n-1}{2}}}{\mathcal K_{\rm err}(\delta)C_{\rm vol}(n)}
\cdot \frac{\delta^2}{\psi(\delta)}.
    \end{aligned}
\end{equation*}
By Claim 4, $\psi(\delta)\geq \delta/(2d)$, hence
\[
\frac{\|N_{\rm loc}^\tau\|}{D_{\rm loc}}
\le
\frac{M(n^2-1)d}{3}\,
\frac{(2\pi)^{\frac{n-1}{2}}(1+dL)^{\frac{n-1}{2}}}{\mathcal K_{\rm err}(\delta)C_{\rm vol}(n)}
\cdot \delta.
\]
Finally, Claim~5 and Claim~2--4 imply $\|N_{\rm tail}\|/D_{\rm loc}=O(e^{-c/\delta})$, which can be absorbed into
$O(\delta^2)$ as $\delta\to 0^+$. This completes the proof. \qed\\
\noindent\textbf{Step 5: Conclusion.}\\
Since $\operatorname{proj}_C(x)=0$ in the chosen coordinates, we have
\[
\|p_\delta(x)-\operatorname{proj}_C(x)\|=\|p_\delta(x)\|\le |p_\delta^\nu(x)|+\|p_\delta^\tau(x)\|.
\]
By Claim~8 (together with Claims~1,5,6 to control the tail and the exponentially small terms), we obtain
\[
|p_\delta^\nu(x)|\le \left(\frac{L(n-1)}{2}+\frac{1}{d}\right)\delta+O(\delta^{3/2}).
\]
By Claim~10,
\[
\|p_\delta^\tau(x)\|\le C_\tau\,\delta+O(\delta^2),
\qquad
C_\tau=\frac{M(n^2-1)d}{3}\,
\frac{(2\pi)^{\frac{n-1}{2}}(1+dL)^{\frac{n-1}{2}}}{C_{\rm vol}(n)}.
\]
Combining the two estimates yields the theorem. \qed


\end{document}